\documentclass[11pt]{amsart}
\usepackage{amssymb, latexsym, mathrsfs, color}
\usepackage[colorlinks=true, pdfstartview=FitV,linkcolor=blue,citecolor=blue,urlcolor=blue]{hyperref}

    \advance \topmargin by -\headheight
    \advance \topmargin by -\headsep

    \evensidemargin \oddsidemargin
\newtheorem {theorem}    {Theorem}[section]

\newtheorem {lemma}      [theorem]    {Lemma}
\newtheorem {corollary}  [theorem]    {Corollary}

\newtheorem {conjecture} [theorem] {Conjecture}

\theoremstyle{definition}
\newtheorem{remark}[theorem]{Remark}

\def\Z{\mathbb{Z}}
\def\R{\mathbb{R}}
\def\C{\mathbb{C}}

\def\Mg{\mathfrak{g}}
\def\a{\alpha}

\numberwithin{equation}{section}

\begin{document}

\title[Eisenstein series on rank $2$ hyperbolic Kac--Moody groups]{Eisenstein series on rank $2$ hyperbolic Kac--Moody groups}

\date{\today}

\author[Lisa Carbone]{Lisa Carbone$^{\dagger}$}
\address{Department of Mathematics \\ Rutgers University \\ Piscataway, NJ 08854-8019.}
\email{carbonel@math.rutgers.edu}
\thanks{$^{\dagger}$This work was supported in part by NSF grant \#DMS--1101282.}

\author[Kyu-Hwan Lee]{Kyu-Hwan Lee$^{\star}$}
\thanks{$^{\star}$This work was partially supported by a grant from the Simons Foundation (\#318706).}
\address{Department of
Mathematics, University of Connecticut, Storrs, CT 06269, U.S.A.}
\email{khlee@math.uconn.edu}

\author[Dongwen Liu]{Dongwen Liu$^{\star\star}$}
\thanks{$^{\star\star}$This work was partially supported by NSFC \#11201384.}
\address{Department of
Mathematics, University of Connecticut, Storrs, CT 06269, U.S.A.}
\email{dongwen.liu@uconn.edu}


\begin{abstract}  We define Eisenstein series on rank $2$ hyperbolic Kac--Moody  groups over $\R$, induced from quasi--characters. We prove convergence of the constant term and hence the almost everywhere convergence of the Eisenstein series. We define and calculate the degenerate Fourier coefficients. We also consider Eisenstein series induced from cusp forms and show that these  are entire functions.

\end{abstract}

\maketitle

\section{Introduction}\label{sec:intro}

After being developed by Langlands \cite{La1, La2} in great generality,  the theory of Eisenstein series has played a fundamental role in the formulation of  the Langlands functoriality conjecture and in the study of $L$-functions by means of the Langlands--Shahidi method. Eisenstein series also appear in many other places throughout number theory and representation theory. The scope of applications is being extended to geometry and mathematical physics. On the other hand, since we have seen  many successful generalizations of finite dimensional constructions to infinite dimensional  Kac--Moody groups \cite{K, Ku}, it is a natural question to ask whether one can generalize the theory of Eisenstein series to Kac--Moody groups. Such an attempt is not merely for the sake of generalization. Even though it is hypothetical for the present, a satisfactory theory of Eisenstein series on Kac--Moody groups would have significant impact on some of the central problems in number theory \cite{BFH,Sh}.

In pioneering work,  H. Garland developed a theory of Eisenstein series for the affine Kac--Moody groups over $\R$ in a series of papers \cite{G99, G04, G06, GMS1, GMS2, GMS3, GMS4, G11}, and he established absolute convergence and meromorphic continuation. The absolute convergence result has been generalized to the case of number fields by D. Liu \cite{Li}. In a recent preprint \cite{GMP}, Garland, Miller and Patnaik  showed that Eisenstein series  induced from cups forms  are entire functions. Garland's idea was extended to the function field case by Kapranov \cite{Ka} through geometric methods and was systematically developed by Patnaik \cite{P}. An algebraic approach to this case was made by K.-H. Lee and Lombardo \cite{LL}. Braverman and Kazhdan's recent preprint \cite{BK} announces more results in the function field case.

The purpose of this paper is to construct Eisenstein series on rank $2$ hyperbolic Kac--Moody groups over $\R$, generalizing Garland's work in the affine case. The rank $2$ hyperbolic Kac--Moody groups form the first family beyond the affine case. However, contrary to the affine case, our understanding of hyperbolic Kac--Moody groups (and algebras) is far from being complete. In particular, information regarding imaginary root multiplicities of hyperbolic Kac--Moody algebras is limited. A recent survey on this topic can be found in \cite{CFL}.

Nevertheless, we have the necessary information to construct Eisenstein series induced from quasi--characters on rank $2$ hyperbolic Kac--Moody groups and to prove their almost everywhere convergence, thanks to the works of Lepowsky and Moody \cite{LM}, Feingold \cite{Fein} and Kang and Melville \cite{KM}. We can also prove entirety of the Eisenstein series induced from cusp forms. Indeed, one of the benefits of working in the Kac--Moody group rather than its Kac--Moody algebra, is that the group is generated by root groups corresponding to only `real' roots. The `real' part of  Kac--Moody  is sufficiently well understood and carries many properties similar to finite dimensional simple Lie algebras \cite{CG1}.

We assume that $G$ is a rank $2$ hyperbolic Kac--Moody group attached to a symmetric $2\times 2$ generalized Cartan matrix, and we define Eisenstein series on the `arithmetic' quotient $K(G_\R)\backslash G_\R/G_\Z$, where $K=K(G_\R)$ is the unitary form of $G$, an infinite dimensional analogue of a maximal compact subgroup. Our method is to choose a quasi--character $\nu$ on a Borel subgroup and then extend it to the whole of $G_\R$ via Iwasawa decomposition $G_\R=KA^+N$, which is given uniquely. Here $A^+\cong (\mathbb{R}^+)^{rank(G)}$ is an abelian subgroup and $N$ is the completion of the subgroup generated by all positive  real root groups.

We then average over an appropriate quotient of $G_\Z$  to obtain a $G_\Z$--invariant function $E_{\nu}(g)$ on $K\backslash G_\R/G_\Z$.  Our first main result is:
\begin{theorem}
Assume that $\nu$ satisfies Godement's criterion, and consider the cone \[A'=\{a \in A^+ : a^{\alpha_i} <1, i=1,2 \}.\] Then for any compact subset $A_c'$ of $A'$, there is a measure zero subset $N_0$ of $N$ such that $E_\nu(g)$ converges absolutely for $g\in KA_c' N'$, where $N'=N-N_0$.
\end{theorem}

Although the idea of the proof is similar to that of \cite{G04}, our proof heavily depends on a concrete description of root systems of rank $2$ hyperbolic Kac--Moody algebras. We compute the constant term of the series $E_{\nu}(g)$ and show that the constant term is absolutely convergent, which implies almost everywhere convergence of the series.  We conjecture that the Eisenstein series actually converges everywhere under a weaker condition than Godement's criterion (See Conjecture \ref{conj-every}). As the argument in \cite{G06} does not generalize to the hyperbolic case, the conjecture seems out of reach at the current time.

We also calculate Fourier coefficients of the Eisenstein series in Section \ref{sec-df}. Let $\psi$ be a non-trivial character of $N/(G_\Z\cap N)$.  Then we can write $\psi=\psi_1\psi_2$, where
$\psi_i$ corresponds to the simple root $\alpha_i$ for $i=1,2$.  We call $\psi$ {\em generic} if each $\psi_i$ is non-trivial for $i=1,2$.
We first show that the Fourier coefficients attached to generic characters vanish (Lemma \ref{vanishing}). Then we consider
characters of the form $\psi=\psi_i$ (i.e. either $\psi_1$ or $\psi_2$ is trivial) and compute the corresponding Fourier coefficients.
The resulting formula is an infinite sum of products of the $n$-th Whittaker coefficient of the analytic Eisenstein series on $SL_2$ and quotients of the completed Riemann zeta function (Theorem \ref{thm-df}).

The next main result  is the entirety of the Eisenstein series $E_{s,f}(g)$ induced from a cusp form $f$. Our approach is similar to that of Garland, Miller and Patnaik in \cite{GMP}; however, our method requires us  to use information about the structure of  the root system of $G$. We obtain:

\begin{theorem} Let $f$ be an unramified cusp form on $SL_2$. For any compact subset $A_c'$ of $A'$, there is a measure zero subset $N_0$ of $N$ such that
$E_{s,f}(g)$ is an entire function of $s\in \mathbb{C}$ for $g\in KA_c'N'$, where $N'=N-N_0$.
\end{theorem}

As mentioned earlier, rank $2$ hyperbolic Kac--Moody algebras and groups form the first family beyond the affine case. It would be interesting to generalize the results of this paper to other hyperbolic Kac--Moody groups, for example, to the Kac--Moody group corresponding to the Feingold and Frenkel's rank $3$ hyperbolic Kac--Moody algebra \cite{FF}. Actually, in a subsequent paper \cite{CGLLM} with Miller and Garland, we will prove that for general Kac--Moody groups, the Eisenstein series  $E_{\nu}(g)$ converges almost everywhere in  the full region satisfying the Godement's criterion $\mathrm{Re}\,\nu(h_{\alpha_i})< -2$. It will be very exciting to see further developments toward a satisfactory theory of Eisenstein series on Kac--Moody groups.

\section{Rank $2$ hyperbolic Kac--Moody algebras and $\mathbb Z$-forms} \label{Zform}

Let $\mathfrak g=\Mg_{\C}$ be the rank $2$  hyperbolic Kac--Moody algebra associated with the symmetric generalized Cartan matrix
\[ \begin{pmatrix}
~2 & -m \\
 -m & ~2 \end{pmatrix} , \qquad  m \ge 3 .\]  Let $\mathfrak{h}=\mathfrak{h}_{\C}$ be a Cartan subalgebra.  Let $\Phi$ be the corresponding root system and let $\Phi_{\pm}$ denote the positive and negative roots respectively. Let
$$\Mg\ =\ \Mg^-\oplus \mathfrak{h} \oplus \Mg^+$$
be the triangular decomposition of $\Mg$, where
$$\Mg^-=\bigoplus_{\a\in\Phi_-}\Mg_{\a},\quad \Mg^+=\bigoplus_{\a\in\Phi_+}\Mg_{\a}.$$

Let $W=W(A)$ be the Weyl group of $\Mg$. We have
 $$W(A)=\langle r_1,r_2\mid r_1^2=1,r_2^2=1\rangle$$ which is the infinite dihedral group
$$W={\Z}/2{\Z} \ast{\Z}/2{\Z} \cong{\Z}\rtimes \{\pm 1\},$$
where $\langle (r_1r_2)\rangle\cong{\Z}$.

 A root $\alpha\in\Phi$ is called a {\it real root} if there exists $w\in W$ such that $w\alpha$ is a
simple root. A root $\alpha$ which is not real is called {\it
imaginary}. We denote by $\Phi^{re}$ the real roots and  $\Phi^{im}$ the
imaginary roots.

Set $I=\{ 1,2\}$. We let $\Lambda\subseteq \mathfrak{h}^{\ast}$ be the $\R$--linear span of  the simple roots $\alpha_i$, for $i\in I$, and $\Lambda^{\vee}\subseteq
\mathfrak{h}$ be the $\R$--linear span of the simple coroots
$h_{\alpha_i}$, for $i\in I$. Let $e_i=e_{\alpha_i}$ and $f_i=f_{\alpha_i}$ be root vectors in $\Mg$ corresponding to  simple roots $\alpha_i$, $i \in I$.
Let ${\mathcal U}_\C$, ${\mathcal U}^+_\C$ and ${\mathcal U}^-_\C$ be the universal enveloping algebras of
$\mathfrak{g}$, $\mathfrak{g}^{+}$ and $\mathfrak{g}^{-}$ respectively. We define the following $\mathbb Z$-subalgebras: Let
\begin{enumerate}
\item ${\mathcal U}^+_{{\Z}}\subseteq {\mathcal U}^+_{{\C}}$ be the ${\Z}$-subalgebra generated by $\dfrac{e_i^n}{n!}$ for $i\in I$
and
$n\geq 0$,

\item ${\mathcal U}^-_{{\Z}}\subseteq {\mathcal U}^-_{{\C}}$ be the ${\Z}$-subalgebra generated by $\dfrac{f_i^{n}}{n!}$ for $i\in I$
and
$n\geq 0$,

\item ${\mathcal U}^0_{{\Z}}\subseteq {\mathcal U}(\mathfrak{h}_{\C})$ be the ${\Z}$-subalgebra generated by $\left (\begin{matrix}
h \\ n\end{matrix}\right )$, for
$h\in\Lambda^{\vee}$ and $n\geq 0$, where
$$\left (\begin{matrix}h\\ n\end{matrix}\right )=\dfrac{h(h-1)\dots (h-n+1)}{n!},$$

\item ${\mathcal U}_{{\Z}}\subseteq {\mathcal U}_{{\C}}$ be the ${\Z}$-subalgebra generated by $\dfrac{e_i^{n}}{n!}$,
$\dfrac{f_i^{n}}{n!}$ for $i\in I$ and $\left (\begin{matrix}
h\\ n\end{matrix}\right )$, for
$h\in\Lambda^{\vee}$ and
$n\geq 0$.
\end{enumerate}
 It follows (\cite{Ti1}) that ${\mathcal U}_{\Z}$ is a {\em ${\Z}$-form} of ${\mathcal U}_{\C}$, i.e. the
canonical map $$\mathcal{U}_{\Z}\otimes_{\Z}\C\longrightarrow \mathcal{U}_\C$$ is bijective.

Recall that $\mathfrak{g}^+ =\bigoplus_{\alpha\in\Phi_+}\mathfrak{g}_{\alpha}$. Let $V$ be a  representation of $\Mg$. Then $V$ is called a {\it highest weight representation} with highest weight $\lambda\in\mathfrak{h}^{\ast}$ if there exists $0\neq v_{\lambda}\in V$ such that
$$\mathfrak{g}^{+}(v_{\lambda})=0,$$
$$h(v_{\lambda})={\lambda}(h)v_{\lambda}$$
for $h\in\mathfrak{h}$ and
$$V=\mathcal {U}_\C \cdot v_{\lambda} .$$
Since $\mathfrak{g}^{+}$ annihilates $v_{\lambda}$ and $\mathfrak{h}$ acts as scalar multiplication on $v_{\lambda}$, we have
$$V=\mathcal {U}^-_\C \cdot v_{\lambda} .$$
We write $V=V^{\lambda}$ for the unique irreducible highest weight module with highest weight $\lambda$. We will assume that  $V=V^{\lambda}$ is integrable.

  We shall construct a lattice $V_{\Z}$ in $V$ by taking the orbit of a highest weight vector $v_{\lambda}$ under $\mathcal{U}_{\Z}$. We have $\mathcal{U}^+_{\Z} \cdot v_{\lambda}=\Z v_{\lambda}$
since all elements of ${\mathcal U}^+_{{\Z}}$ except for 1 annihilate $v_{\lambda}$. Also ${\mathcal U}^0_{{\Z}}$ acts as scalar multiplication on $v_{\lambda}$ by a $\Z$-valued scalar, since  $\left (\begin{matrix}
h\\ n\end{matrix}\right )$ for
$h\in\Lambda^{\vee}$ and $n\geq 0$ acts on $v_{\lambda}$ as
$$\left (\begin{matrix}\lambda(h)\\ n\end{matrix}\right )=\dfrac{\lambda(h)(\lambda(h)-1)\dots (\lambda(h)-n+1)}{n!}\in\Z.$$
Thus
$$\mathcal{U}^0_{\Z} \cdot v_{\lambda}=\Z v_{\lambda}, \qquad
 \mathcal{U}_{\Z}\cdot v_{\lambda}=\mathcal{U}^-_{\Z}\cdot (\Z v_{\lambda})=\mathcal{U}^-_{\Z}\cdot v_{\lambda}.$$

Let $\alpha$ be any real root and let $e_{\alpha}$ and $f_{\alpha}$ be root vectors corresponding to $\alpha$.  Then
\[  \dfrac{f_{\alpha}^n}{n!}v_{\lambda}\in V_{\lambda-n\alpha}.\]
For a weight $\mu<\lambda$ we have
\[ \dfrac{e_{\alpha}^n}{n!}v_{\mu}\in V_{\mu+n\alpha}.\]
   We set
$$V_\Z\  =\ \mathcal{U}_{\Z}\cdot v_{\lambda}\ =\ \mathcal{U}^-_{\Z}\cdot v_{\lambda}$$
Then $V_{\Z}$  is a lattice in  $V_{\C}$ and a ${\mathcal U}_\Z$-module.

  For each weight $\mu$ of  $V$, let  $V_{\mu}$ be the corresponding weight
space, and we set
$$V_{\mu,{\Z}}\ =\  V_{\mu}\cap V_{\Z}.$$
We have
$$V_{\Z}\ =\ \oplus_{\mu}V_{\mu,{\Z}},$$
where the sum is taken over the weights of $V$. Thus $V_{\Z}$ is a direct sum of its weight spaces. We set
$$V_{\mu,\R}\ =\ \R\otimes_{\Z} V_{\mu,\Z}$$
so that
$$V_{\R}\ :=\ \R\otimes_{\Z} V_{\Z}\ =\ \oplus_{\mu}V_{\mu,{\R}}.$$

  For each weight
$\mu$ of $V$, we have $\mu=\lambda-(k_1 \alpha_1 +k_2\alpha_2)$, where $\lambda$ is the highest weight and  $k_i\in{\Z}_{\geq 0}$. Define the {\it depth} of
$\mu$ to be
$$depth(\mu)=k_1+k_2.$$
A basis $\Psi=\{v_1,v_2,\dots\}$ of $V$ is called {\it coherently ordered relative to depth} if

 (1) $\Psi$ consists of weight vectors;

 (2) If $v_i\in V_{\mu}$, $v_j\in V_{\mu'}$ and $depth(\mu')\ >\ depth(\mu)$, then $j>i$;

 (3) $\Psi\cap V_{\mu}$ consists of an interval $v_k,v_{k+1},\dots,v_{k+m}$.

\begin{theorem} [\cite{G78}] The lattice $V_{\Z}$ has a coherently ordered $\Z$-basis $\{v_1,v_2,\dots\}$ where $v_i\in V_{\Z}$,  $v_i=\xi_i v_{\lambda}$, for some $\xi_i\in \mathcal{U}_{\Z}$.  Let $w_i=k_i\otimes v_i$, $k_i\in \R \setminus \{0\}$. Then the set $\{w_1,w_2,\dots \}$ is a  coherently ordered basis for $V_{\R}$. Any vector in $V_{\Z}$ has an integer valued norm relative to a Hermitian inner product $\langle , \rangle$ on $V$.
\end{theorem}


\section{The Kac--Moody group $G$ and Iwasawa decomposition} \label{Group}

Our next step is to construct our Kac--Moody group $G$ over $\R$. The construction below can be used to construct $G$ over any field $F$ \cite{CG1}.   As before, let $V$ be  an integrable  highest weight module for $\mathfrak{g}_\C$.  Then the simple root vectors $e_i$
and $f_i$ are locally nilpotent on  $V$.

 We let $V_{\Z}$ be a $\Z$-form of $V$ as in Section~\ref{Zform}. Since $V_{\Z}$ is a $\mathcal{U}_{\Z}$-module, we have
$$\dfrac{e_i^n}{n!} (V_{\Z})\subseteq V_{\Z} \quad \text{ and } \quad \dfrac{f_i^n}{n!} (V_{\Z})\subseteq V_{\Z} \quad \text{for } \ n\in \mathbb{N}, \ i \in I .$$
Let $G_{\R}=\R\otimes_{\Z} V_\Z$.  For $s,t\in \R$ and $i\in I$,  set
$$\chi_{\alpha_i}(s)\ =\ \sum_{n=0}^\infty s^n \dfrac{e_i^n}{n!}\ =\exp(se_i),\qquad \chi_{-\alpha_i}(t)\ =\ \sum_{n=0}^\infty t^n \dfrac{f_i^n}{n!}\ =\exp(tf_i).$$
Then $\chi _{\alpha _{i}}(s),$ $\chi
_{-\alpha
_{i}}(t)$ define elements in $\mathrm{Aut}(V_{\R})$, thanks to the
local
nilpotency of $e_{i},$ $f_{i}$.  More generally, for a real root $\alpha$, we write $\alpha=w\alpha_i$ for $i=1$ or 2 and define 
$$\chi_{\alpha}(s)=\chi_{w\alpha_i}(s)=w\chi_{\alpha_i}(s)w^{-1}\in \mathrm{Aut}(V_{\R}),  \qquad s \in \R.$$ 
For $t \in \R^{\times}$, we set \[ w_{\alpha_i}(t) = \chi_{\alpha_i}(t)  \chi_{-\alpha_i}(-t^{-1}) \chi_{\alpha_i}(t) \quad \text{ for } i \in I, \] and define \[ h_{\alpha_i}(t)= w_{\alpha_i}(t) w_{\alpha_i}(1)^{-1} .\]

  We let $G^0_{\R}$ be the subgroup of $\mathrm{Aut}(V_{\R})$ generated by the linear automorphisms $\chi_{\alpha_i}(s)$ and  $\chi_{-\alpha_i}(t)$ of
$V_{\R}$, for $s,t\in \R$, $i \in I$. That is, we define
$$G^0_{\R}=\langle \exp(se_i), \ \exp(tf_i) : s,t\in \R, \ i \in I \rangle.$$
One can see that $\chi_\alpha(s) \in G^0_\R$ for real roots $\alpha$. 

We choose a coherently basis $\Psi=\{v_1,v_2,\dots\}$ of $V_\R$ and denote by $B^0$ the subgroup of $G^0_\R$ consisting of the elements represented by upper triangular matrices with respect to $\Psi$.
For $t \in \mathbb Z_{>0}$, we let $V_t$ be the span of the $v_s \in \Psi$ for $s \le t$. Then $B^0 V_t \subseteq V_t$ for each $t$. Let $B_t$ be the image of $B^0$ in $\mathrm{Aut}(V_t)$. We then have surjective homomorphisms
\[ \pi_{tt'}: B_{t'} \longrightarrow B_t, \quad t' \ge t. \] We define $B$ to be the projective limit of the projective family $\{ B_t, \pi_{tt'} \}$. 

Next we consider the completion of $G^0_\R$. We can define a topology on $G^0_\R$ as follows: for a base of neighborhoods of the identity, we take sets $V_t$ defined by \[ V_t=\{ g \in G^0_\R : g v_i =v_i, \ i=1, 2, \dots , t \}. \]
Let $U_p= \mathrm{Span} \{ v_1,...,v_p \}$. Let $g=(g_i)$ and $h=(h_i)$ be Cauchy sequences in $G^0_\R$. For any $p$, we can find sufficiently large $q \ge p$ such that $\{g_i^{-1}(U_p) \}$ is contained in $U_q$. Take $N$ sufficiently large so that $g_i$ and $h_i$ represent the restrictions of $g$ and $h$ to $U_q$, respectively, whenever $i>N$. Assume that $i, j >N$. Then $h_i h_j^{-1}  \in V_q$, and we have
\[ g_i h_i h_j^{-1} g_j^{-1} (u) = g_i g_j^{-1} (u) = u \ \text{ for any} \ u \in U_p, \]
since $g_j^{-1}(u) \in U_q$. 
 This proves that $g_i h_i h_j^{-1} g_j^{-1} \in V_p$ whenever $i, j >N$. Therefore $gh=(g_i h_i)$ is Cauchy as well.

Let $G_\R$ be the completion of $G^0_\R$, i.e. the equivalence classes of all Cauchy sequences of $G^0_\R$.  By replacing $\R$ with $F$ in the above construction, we obtain the group $G_{F}$ for  any field $F$.

Note that $B$ is naturally a subgroup of $G_\R$. We define the following subgroups of $G_\R$:

 \begin{enumerate}
 \item  $N=$  the completion of the subgroup generated by all positive real root groups and is a subgroup of $B$,

\item $K=\{k\in G_\R: k\textrm{ preserves the inner product }\langle,\rangle \textrm{ on }V^\lambda_\R\}$, \item $A=\langle h_{\alpha_i}(s) : s\in\mathbb{R}^\times, i \in I\rangle$ \ and \ $A^+=\langle h_{\alpha_i}(s) : s\in\mathbb{R}_+, i \in I\rangle$. 
\end{enumerate}

\begin{theorem}[\cite{DGH}]
We have the Iwasawa decomposition:
\begin{equation}\label{iwasawa}
G_\R= KA^+N
\end{equation} with uniqueness of expression.
\end{theorem}

As in \cite{Ca}, we now define the `$\Z$-form' $G_\Z$ of $G_\R$ in the following way. We set
$$G_\Z=G_\R\cap \mathrm{Aut}(V_{\mathbb{Z}}).$$
Then $$G_\Z=\{\gamma\in G_\R : \gamma\cdot V_{\mathbb{Z}}\subseteq V_{\mathbb{Z}}\}.$$

\begin{remark}
For a discussion on dependence on the choice of $V$ and
$V_{\mathbb Z}$, we refer the reader to \cite{Ca}. In this paper, we work with fixed $V$ and $V_{\mathbb Z}$.
\end{remark}

\section{Eisenstein series on rank 2 hyperbolic Kac--Moody groups}\label{KMEseries}

 Let $g=k_ga_gn_g\in G_\R$ be the Iwasawa decomposition according to (\ref{iwasawa}). Let $\nu: A^+ \to \mathbb{C}^\times$ be a quasi--character and define
\[
\Phi_{\nu}:G_\R\to {\mathbb C}^{\times}
\]
to be the function
\[
\Phi_{\nu}(g)=\nu(a_g)
\]
Then $\nu$ is well defined since the Iwasawa decomposition is unique and $\Phi_{\nu}$ is left $K$-invariant and right $N$-invariant.  For convenience, we write $\Gamma=G_\Z$.

 Let $B$ denote the minimal parabolic subgroup of  $G_\R$. Relative to a coherently ordered basis $\Psi$ for $V^{\lambda}_{\Z}$, $\Gamma$ has a representation in terms of infinite matrices with integral entries.
 Define the Eisenstein series on $G_\R$ to be the infinite formal sum
$$E_{\nu}(g)\quad:=\quad \sum_{\gamma\in \Gamma/\Gamma\cap{B}}\quad
\Phi_{\nu}(g\gamma).$$

 Recall that $\mathfrak{h}$ is the Lie algebra of $A$ and that $h_{\alpha_i}$, $i\in I$, are the simple coroots.
We say that $\nu$ satisfies {\em Godement's criterion} if
\[
\textrm{Re }\nu(h_{\alpha_i})<-2,\quad i\in I.
\]
We do not expect that the Eisenstein  series will  be convergent over the whole space $K\backslash G_\R/\Gamma$ but rather a subspace, where if $G_\R=KAN$ is decomposed in terms of the Iwasawa decomposition, the $A$-component  is replaced by the `group' corresponding to the Tits cone to obtain $G'_\R=KA'N$. Godement's criterion places the sum in a cone in the region Re $\nu(h_{\alpha_i})\leq -2$ for each $i$. In the next section,  we deduce almost everywhere convergence of the Eisenstein series from convergence of the constant term.

We will not comment on the difficult question of meromorphic continuation to the whole complex plane here. However, we will prove in \S7 that cuspidal Eisenstein series are entire.

\section{Convergence of the constant term}

 In this section we prove convergence of the constant term and thus almost everywhere convergence of the Eisenstein series itself.
Assume first that $\nu: A^+ \to \mathbb{C}^\times$  is real valued and positive. Then we may interpret the infinite sum $E_\nu(g)$ as a function taking values in $\mathbb{R}_+\cup\{\infty\}$. The function $E_\nu$ may be regarded as a function on
\[
K\backslash G_\R/ (\Gamma\cap N) \cong A^+\times N/ (\Gamma\cap N).
\]

Under the identification $\mathbb{R}_+^2\cong A^+$:
\[
(x_1,x_2)\mapsto h_{\alpha_1}(x_1)h_{\alpha_2}(x_2),
\]
we have the measure $da$ on $A^+$, corresponding to the measure
\[
\frac{dx_1}{x_1}\frac{dx_2}{x_2}
\]
on $\mathbb{R}_+^2$.

As in \cite{G04} we know that $N/(\Gamma\cap N)$ is the projective limit of a projective family of finite-dimensional compact nil-manifolds and thus admits
a projective limit measure $dn$, which is a left $N$-invariant probability measure. More precisely, recall from  Section \ref{Group} that we have a coherently basis $\Psi$ and the spaces $V_t$ for $t \in \mathbb Z_{>0}$.
Denote by $N^0$ the subgroup of $G^0_\R$ consisting of the elements represented by upper triangular matrices with respect to $\Psi$ with $1$'s on the diagonal.
Let $N_t$ be the image of $N^0$ in $\mathrm{Aut}(V_t)$. Then the group $N$ is the the projective limit of the projective family $\{ N_t, \pi_{tt'} \}$. We also have the canonical projection $\pi_t : N \rightarrow N_t$ for each $t$. Consider $\Gamma_t:= \pi_t(\Gamma \cap N)$. Then the space  $N_t/\Gamma_t$ is a finite-dimensional compact nil-manifold, on which we define a probability measure compatible with the projections $\pi_{tt'}$. Now we obtain a projective limit measure $dn$ on $N/(\Gamma \cap N)$.

 We define for all $g\in G_\R$ the constant term
\[
E^\sharp_\nu(g)=\int_{N/(\Gamma\cap N)}E_\nu(gn)dn
\]
which is left $K$-invariant and right $N$-invariant. In particular $E^\sharp_\nu(g)$ is determined by the $A^+$-component of $g$ in the Iwasawa decomposition.
Let $\rho\in \frak{h}^\ast$ satisfying $\rho(h_{\alpha_i})=1$, $i\in I$. Then
\[
\rho=\frac{\alpha_1+\alpha_2}{2-m}.
\]
Applying the Gindikin-Karpelevich formula,  a formal calculation as in \cite{G04} yields that for $a\in A^+$
\[
E^\sharp_\nu(a)=\sum_{w\in W}a^{w(\nu+\rho)-\rho}c(\nu,w),
\]
where
\[
c(\nu,w)=\prod_{\alpha\in \Phi_+\cap w^{-1}\Phi_-}\frac{\xi(-(\nu+\rho)(h_\alpha))}{\xi(1-(\nu+\rho)(h_\alpha))},
\]
and $\xi(s)$ is the completed Riemann zeta function
\[
\xi(s)=\pi^{-s/2}\Gamma(s/2)\prod_p\frac{1}{1-p^{-s}}.
\]

 Before proving the convergence of the constant term, let us first give some preliminaries for the structure of the root system of $\mathfrak g$, following \cite{KM}. Let
\begin{equation}\label{root}
\gamma=\frac{m+\sqrt{m^2-4}}{2},
\end{equation}
which is a root of the polynomial $x^2-mx+1$. Let $r_1$, $r_2$ be the simple reflections corresponding to the simple roots $\alpha_1$, $\alpha_2$. Then the Weyl group $W$ is generated by $r_1$, $r_2$ subject to
the relations $r_1^2=r_2^2=1$, and has an explicit description
\[
W=\{1, r_1(r_2r_1)^n, r_2(r_1r_2)^n, (r_1r_2)^{n+1}, (r_2r_1)^{n+1} : n\geq 0\}.
\]
We introduce a sequence $\{A_n\}$ defined by
\[
A_0=0, \quad A_1=1, \quad A_{n+2}=mA_{n+1}-A_n+1,\quad n\geq 0.
\]
Then we have the explicit formula
\begin{equation}\label{an}
A_n=\frac{\gamma^{2n+1}-\gamma^n(1+\gamma)+1}{\gamma^{n-1}(\gamma+1)(\gamma-1)^2}= \frac{\gamma^{n+2}}{(\gamma+1)(\gamma-1)^2}+O(1),\quad n\geq 0.
\end{equation}
We also need another sequence
\begin{equation}\label{bn}
B_n=A_n-A_{n-1}=\frac{\gamma^{2n}-1}{\gamma^{n-1}(\gamma^2-1)}= \frac{\gamma^{n+1}}{\gamma^2-1}+o(1),\quad n\geq 0.
\end{equation}
Then we have the following formulas for the actions of  $ r_1(r_2r_1)^n$ and $(r_1r_2)^{n+1}$ on simple roots:
\begin{eqnarray}
\label{w1}&&\left\{\begin{array}{l}  r_1(r_2r_1)^n\alpha_1=-B_{2n+1}\alpha_1-B_{2n}\alpha_2, \\ r_1(r_2r_1)^n\alpha_2= B_{2n+2}\alpha_1+B_{2n+1}\alpha_2.\end{array}\right.\\
\label{w2}&&\left\{\begin{array}{l} (r_1r_2)^{n+1}\alpha_1=B_{2n+3}\alpha_1+B_{2n+2}\alpha_2, \\ (r_1r_2)^{n+1}\alpha_2=-B_{2n+2}\alpha_1-B_{2n+1}\alpha_2.\end{array}\right.
\end{eqnarray}
Switching $\alpha_1$ and $\alpha_2$, we may also obtain the similar actions of $r_2(r_1r_2)^n$ and $(r_2r_1)^{n+1}$. Regarding $w\rho-\rho$ we have
\begin{equation}\label{wrho}
\left\{\begin{array}{l}
r_1(r_2r_1)^n\rho-\rho=-A_{2n+1}\alpha_1-A_{2n}\alpha_2,\\
 r_2(r_1r_2)^n\rho-\rho=-A_{2n}\alpha_1-A_{2n+1}\alpha_2,\\
(r_1r_2)^{n+1}\rho-\rho=-A_{2n+2}\alpha_1-A_{2n+1}\alpha_2,\\
(r_2r_1)^{n+1}\rho-\rho=-A_{2n+1}\alpha_1-A_{2n+2}\alpha_2.
\end{array}\right.
\end{equation}
\begin{theorem}
Assume that $\nu$ satisfies Godement's criterion.
Then the constant term $E^\sharp_\nu(g)$ converges absolutely for $g\in KA'N$, where $A'$ is the cone
\[
A'=\{a\in A^+: a^{\alpha_i}<1, i\in I\}.
\]
\end{theorem}

\begin{proof} We may without loss of generality assume that $\nu(h_{\alpha_i})$ has real values, $i\in I$.
Then we may write
\[
\nu= s_1\alpha_1 + s_2\alpha_2
\]
where $s_1, s_2\in\mathbb{R}$. Godement's criterion then reads
\begin{equation}\label{ineq}
\nu(h_{\alpha_1})=2s_1-ms_2<-2,\quad \nu(h_{\alpha_2})=2s_2-ms_1<-2.
\end{equation}
In particular we have $s_1, s_2>0$. Let us consider a typical term
\[a^{w(\nu+\rho)-\rho}c(\nu,w)\] in $E^\sharp_\nu(a)$, where $a\in A'$.
By symmetry we only need to consider $w=r_1(r_2r_1)^n$ and $w=(r_1r_2)^{n+1}$, $n\geq 0$.

For $w=r_1(r_2r_1)^n$, by (\ref{w1}) and (\ref{wrho}) we have
\[
w(\nu+\rho)-\rho=(-s_1B_{2n+1}+s_2 B_{2n+2}-A_{2n+1})\alpha_1+(-s_1B_{2n}+s_2B_{2n+1}-A_{2n})\alpha_2.
\]
By (\ref{an}) and (\ref{bn}) we have
\begin{eqnarray*}
&&-s_1B_{2n+1}+s_2 B_{2n+2}-A_{2n+1}\\
&=& (\gamma s_2 -s_1 ) \frac{\gamma^{2n+2}}{\gamma^2-1} - \frac{\gamma^{2n+3}}{(\gamma+1)(\gamma-1)^2}+O(1)\\
&=& (\gamma s_2 - s_1 - \frac{\gamma}{\gamma-1})\frac{\gamma^{2n+2}}{\gamma^2-1}+O(1).
\end{eqnarray*}
From (\ref{ineq}) it follows that
\begin{eqnarray*}
\gamma s_2 - s_1 &=& \frac{m\gamma-2}{4-m^2}(2s_1-ms_2)+\frac{2\gamma-m}{4-m^2}(2s_2-ms_1)\\
&> & \frac{2(m\gamma-2+2\gamma-m)}{m^2-4}\\
&=& \frac{2(\gamma-1)}{m-2} =\frac{2\gamma}{\gamma-1},
\end{eqnarray*}
where the last equation follows from $\gamma^2-m\gamma+1=0$.
If we introduce a constant
\[
C_\nu= (\gamma s_2 - s_1 - \frac{\gamma}{\gamma-1})\frac{1}{\gamma^2-1}>0,
\]
then
\[
r_1(r_2r_1)^n(\nu+\rho)-\rho=(C_\nu\gamma^{2n+2}+O(1))\alpha_1+ (C_\nu\gamma^{2n+1}+O(1))\alpha_2.
\]
Similarly, we have
\[
(r_1r_2)^{n+1}(\nu+\rho)-\rho=(D_\nu\gamma^{2n+3}+O(1))\alpha_1+ (D_\nu\gamma^{2n+2}+O(1))\alpha_2,
\]
where
\[
D_\nu=(\gamma s_1 - s_2 - \frac{\gamma}{\gamma-1})\frac{1}{\gamma^2-1}>0.
\]

By Godement's criterion we have
\[
-\textrm{Re }(\nu+\rho)(h_{\alpha_i})>1
\]
for $i\in I$. Then there exists $\varepsilon>0$ such that
\[
-\textrm{Re }(\nu+\rho)(h_\alpha)>1+\varepsilon
\]
for any real root $\alpha$. Using the properties of Riemann zeta functions we can find a constant $C_\varepsilon>0$ depending on
$\varepsilon$ such that
\[
|c(\nu, w)|< C_\varepsilon^{\ell(w)}
\]
where $\ell(w)$ is the length of $w$.

Since $a\in A'$, we have $a^{\alpha_i}<1$, $i\in I$. Combining above estimates we see that the series defining $E^\sharp_\nu(a)$ converges absolutely.
\end{proof}

\begin{corollary}
Assume that $\nu$ satisfies Godement's criterion. Then for any compact subset $A_c'$ of $A'$, there is a measure zero subset $N_0$ of $N$ such that $E_\nu(g)$ converges absolutely for $g\in KA_c' N'$, where $N'=N-N_0$.
\end{corollary}

We propose the following conjecture, which weakens Godement's criterion and asserts everywhere convergence instead of almost everywhere convergence.

\begin{conjecture} \label{conj-every}
$E_\nu(g)$ converges absolutely for $g\in KA'N$ and $\nu$ satisfying $\mathrm{Re}~\nu(h_{\alpha_i})<-1$, $i\in I$.
\end{conjecture}

\section{Fourier coefficients} \label{sec-df}

 In this section we shall define and calculate the Fourier coefficients of $E_\nu(g)$. To facilitate the computation, we work with adelic groups. Let $\mathbb{A}=\mathbb{R}\times\prod'_p\mathbb{Q}_p$ and $\mathbb{I}=\mathbb{A}^\times$ be the adele ring and idele group of the rational number field $\mathbb{Q}$, respectively.
According to Section \ref{Group}, for any prime $p$ we have the group $G_{\mathbb{Q}_p}\subset \mathrm{Aut}(V_{\mathbb{Q}_p})$, and we let
$K_p\subset G_{\mathbb{Q}_p}$ be the subgroup $K_p=\{g\in G_{\mathbb{Q}_p}: g \cdot V_{\mathbb{Z}_p}=V_{\mathbb{Z}_p}\}$.
Let $G_{\mathbb{A}}=G_\R\times\prod'_p G_{\mathbb{Q}_p}$ and $G_{\mathbb{A}_f}=\prod_p'G_{\mathbb{Q}_p}$ (the adele and finite adele groups respectively)
be the restricted products with respect to the family of subgroups $K_p$. Note that we have the diagonal embedding
$\iota: G_{\mathbb{Q}}\hookrightarrow G_\R\times\prod_p G_{\mathbb{Q}_p}$.
Set $\Gamma_{\mathbb{Q}}=\iota^{-1}(G_\mathbb{A})$ and $K_\mathbb{A}=K\times\prod_p K_p$.

\medskip  We shall extend the definition of $E_\nu(g)$ to $g\in G_\mathbb{A}$. For each prime $p$ we have an Iwasawa decomposition \cite{DGH}
\[
G_{\mathbb{Q}_p}=K_p A_{\mathbb{Q}_p}N_{\mathbb{Q}_p},
\]
where $A_{\mathbb{Q}_p}$ is generated by $h_{\alpha_i}(s)$, $i=1,2$, $s\in\mathbb{Q}_p^\times$, and $N_{\mathbb{Q}_p}$ is generated by
$\chi_\alpha(s)$, $\alpha\in\Phi^{re}_+$, $s\in\mathbb{Q}_p$. From the local Iwasawa decompositions we have
\[
G_\mathbb{A}=K_\mathbb{A} A_\mathbb{A} N_\mathbb{A}.
\]
If $\iota=(\iota_\infty, \iota_p)\in\mathbb{I}$ is an idele, define the usual norm $|\iota|$ of $\iota$ by
\[
|\iota|=|\iota_\infty|\prod_p |\iota_p|_p.
\]
An element $a\in A_\mathbb{A}$ can be decomposed as $a=h_{\alpha_1}(s_1)h_{\alpha_2}(s_2)$, $s_1, s_2\in\mathbb{I}$. We let $|a|\in A^+$ be the element
$h_{\alpha_1}(|s_1|)h_{\alpha_2}(|s_2|)$. Let $g\in G_\mathbb{A}$ be decomposed as
\[
g= k_g a_g n_g,\quad k_g\in K_\mathbb{A}, ~a_g\in A_\mathbb{A}, ~ n_g\in N_\mathbb{A}.
\]
Note that $|a_g|$ is uniquely determined by $g$, although above decomposition is not unique. Then we may define
\[
\Phi_\nu(g)=|a_g|^\nu.
\]
The Eisenstein series is defined by
\[
E_\nu(g)=\sum_{\gamma\in \Gamma_\mathbb{Q}/\Gamma_\mathbb{Q}\cap B_\mathbb{Q}}\Phi_\nu(g\gamma), \quad g \in \mathbb A.
\]
When $g\in G_\R$ this coincides with our previous definition, since $\Gamma_\mathbb{Q}/\Gamma_\mathbb{Q}\cap B_\mathbb{Q}\cong
\Gamma/\Gamma\cap B$.

\medskip  For a positive real root $\alpha$, let $U_\alpha$
be the root subgroup $\{\chi_\alpha(u): u\in \mathbb{R}\}$.

Let $\psi$ be a non-trivial character of $N/(\Gamma\cap N)$. Then we have $\psi=\psi_1\psi_2$, where
$\psi_i$ is a character of $U_{\alpha_i}/\Gamma\cap U_{\alpha_i}$. This follows from the fact that $N/[N,N]\cong U_{\alpha_1}\times U_{\alpha_2}$.
We extend $\psi$ to a character of $N_\mathbb{A}/N_\mathbb{Q}$, where $N_\mathbb{Q}:=\Gamma_\mathbb{Q}\cap N_\mathbb{A}$ and
define the $\psi$-th Fourier coefficient of $E_\nu(g)$ along $B$ by
\[
E_{\nu,\psi}(g)=\int_{N/(\Gamma\cap N)}E_\nu(gn)\bar{\psi}(n)dn=\int_{N_\mathbb{A}/N_\mathbb{Q}}E_\nu(gn)\bar{\psi}(n)dn.
\]
Then $E_{\nu,\psi}(g)$ is a Whittaker function on $G$, that is, a function $W$ satisfying the relation
\[
W(g n) = \psi(n)W(g),
\]
for each $n\in N$.

We call $\psi$ {\em generic} if each $\psi_i$ is non-trivial for $i=1,2$. Then we have the following
vanishing result for generic characters, which in fact holds generally for infinite-dimensional Kac--Moody groups (cf. \cite{Li}).

\begin{lemma} \label{vanishing}
If $\psi$ is generic, then $E_{\nu,\psi}(g)=0$.
\end{lemma}

\begin{proof}
Recall that we have the Bruhat decomposition
\[
G=BWB=\bigsqcup_{w\in W}N_w w B,
\]
where $N_w=\prod_{\alpha\in\Phi_+\cap w\Phi_-}U_\alpha$. Then we have
\[
E_\nu(g)=\sum_{\gamma\in \Gamma/\Gamma\cap B}\Phi_\nu(g\gamma)=\sum_{w\in W}\sum_{\gamma\in  N_{w,\mathbb{Q}}}\Phi_\nu(g\gamma w).
\]
We introduce $N_w'=\prod_{\alpha\in\Phi_+\cap w\Phi_+}U_\alpha$. Then $N=N_wN'_w$ and it follows that
\begin{eqnarray*}
E_{\nu,\psi}(g)&=&\sum_{w\in W}\int_{N_\mathbb{A}/N_\mathbb{Q}}\sum_{\gamma\in  N_{w,\mathbb{Q}}}\Phi_\nu(gn\gamma w)\bar{\psi}(n)dn\\
&=&\sum_{w\in W}\int_{N_\mathbb{A}/N'_{w,\mathbb{Q}}}\Phi_\nu(gn w)\bar{\psi}(n)dn\\
&=&\sum_{w\in W}\int_{N_{w,\mathbb{A}}}\bar{\psi}(n_w)\int_{ N'_{w,\mathbb{A}}/ N'_{w,\mathbb{Q}}}\Phi_\nu(gn_w n_w'w)\bar{\psi}(n_w')dn_w'dn_w.
\end{eqnarray*}
For each $w\in W$, at least one of the two roots $w^{-1}\alpha_i$, $i=1,2$, is positive. Since $\Phi_\nu$ is right $N$-invariant, the inner integral of
the last equation involves a factor $$\int_{U_{\alpha_i, \mathbb{A}}/U_{\alpha_i, \mathbb{Q}}}\bar{\psi}_i(u)du=\int_{U_{\alpha_i}/\Gamma\cap U_{\alpha_i}}\bar{\psi}_i(u)du$$ for some $i$, which is zero by the assumption that $\psi$ is generic.
\end{proof}

Using this lemma, we may assume that $\psi=\psi_1$ or $\psi_2$, and is non-trivial. Any character of $U_{\alpha_i}$, which is trivial on $\Gamma\cap U_{\alpha_i}$, is of the form
\[
\psi_{i,n}: \chi_{\alpha_i}(u)\mapsto e^{2\pi i n u},\quad u\in\mathbb{R}
\]
for some $n\in\mathbb{Z}$.

\medskip  Before we state and prove the main result of this section, let us first recall some Fourier coefficients for $SL_2$.
For $F=\mathbb{R}$ or $\mathbb{Q}_p$, one has the Iwasawa decomposition $SL_2(F)=KAN$, where $K=SO(2,\mathbb{R})$ or $SL_2(\mathbb{Z}_p)$ is a maximal compact subgroup of $SL_2(F)$,
\[
A=\left\{\begin{pmatrix} a & 0 \\ 0 & a^{-1}\end{pmatrix}: a\in \R^\times\right\},\quad N=\left\{\begin{pmatrix} 1 & x \\ 0 & 1\end{pmatrix}: x\in \R\right\}.
\]
Let $g\in SL_2(F)$ be decomposed as
\[
g=k\begin{pmatrix} a_g & 0 \\ 0 & a_g^{-1}\end{pmatrix}n.
\]
For $s\in\mathbb{C}$, define a function $\Phi_s(g)$ on $SL_2(F)$ by
$\Phi_s(g)= |a_g|^{-s}$.
Clearly $\Phi_s$ is well-defined. Given a character $\psi$ of $F$, we shall consider the Fourier coefficient
\begin{equation}\label{sl2}
\int_F \Phi_s\begin{pmatrix} 1 & 0 \\ x & 1\end{pmatrix}\bar{\psi}(x)dx,
\end{equation}
which is convergent for Re $s>0$.
If we write $\begin{pmatrix} a_x & 0 \\ 0 & a_x^{-1}\end{pmatrix}$ for the $A$-component of $\begin{pmatrix} 1 & 0 \\ x & 0\end{pmatrix}$ in the Iwasawa decomposition, then
$|a_x|=\sqrt{1+x^2}$ for $F=\mathbb{R}$ and $|a_x|=\max(1, |x|_p)$ for $F=\mathbb{Q}_p$.

\medskip  The character $\psi(u)=e^{2\pi i u}$ of $\mathbb{R}/\mathbb{Z}$ corresponds to the character
$\psi_\infty \prod_p\psi_p$ of $\prod_p\Z_p\backslash\mathbb{A}/\mathbb{Q}$, where
\begin{eqnarray*}
&& \psi_\infty(x)=e^{2\pi i x}, \quad x\in\mathbb{R}\\
&& \psi_p(x)=e^{-2\pi i (\textrm{fractional part of }x)},\quad x\in\mathbb{Q}_p.
\end{eqnarray*}
Fix $y\in \mathbb{R}_+$ and associate an idele $(y_v)_v\in\mathbb{I}$ by $y_\infty=y$, $y_p=1$. For $n\in\mathbb{Z}$,
$n\neq 0$, we twist the $n$th power of $\psi$ by $y$, i.e. consider the characters $\psi_\infty(nyx)$ of $\mathbb{R}$
and $\psi_p(nx)$ of $\mathbb{Q}_p$.
Then the Fourier coefficients (\ref{sl2}) are given by the following functions. For $F=\mathbb{R}$, Re $s>1$ we have (cf. \cite[pp. 66--67]{Bu})
\begin{eqnarray*}
W^\infty_n(y, s)&=&\int^\infty_{-\infty} (1+x^2)^{-\frac{s}{2}}\bar{\psi}_\infty(nyx)dx\\
&=& 2\pi^{s/2}\Gamma(s/2)^{-1}|ny|^{\frac{s-1}{2}}K_{\frac{s-1}{2}}(2\pi|n|y),
\end{eqnarray*}
where $K_s(y)$ is the $K${\it-Bessel function}, also known as the {\it Macdonald Bessel function}, defined by
 \[
K_s(y)=\frac{1}{2}\int^\infty_0 e^{-y(t+t^{-1})/2}\ t^s\frac{dt}{t},\quad \textrm{Re }s>0.
\]
Write $|n|=\prod_p p^{n_p}$ into the primary decomposition. Then for $p<\infty$, Re $s>1$, we have
\begin{eqnarray*}
W^p_n(s)&=&1+\sum^\infty_{i=1}p^{-is}\int_{p^{-i}\mathbb{Z}_p^\times}\bar{\psi}_p(nx)dx\\
&=&\frac{(1-p^{-s})(1-p^{(n_p+1)(1-s)})}{1-p^{1-s}}.
\end{eqnarray*}
In the above computations we have made use of the Iwasawa decomposition for $SL_2$. Now we form a product
\begin{equation}\label{whittaker}
W_n(y,s)=W^\infty_n(y,s)\prod_p W^p_n(s)=2\sigma_{1-s}(|n|)|ny|^{\frac{s-1}{2}}K_{\frac{s-1}{2}}(2\pi|n|y)\frac{1}{\xi(s)},\quad \textrm{Re }s>1,
\end{equation}
where $\sigma_s$ is the divisor power sum function defined by $\sigma_s(n)=\sum_{d|n}d^{s}$ for $n\in\mathbb{N}$.

\begin{theorem} \label{thm-df}
Assume that $\nu$ satisfies Godement's criterion. Then for $a\in A'$, $i\in I$, $n\in\mathbb{Z}$, $n \neq 0$, one has the $\psi_{i,n}$-th Fourier coefficient
\[
E_{\nu,\psi_{i,n}}(a)=\sum_{w\in W,\ w^{-1}\alpha_i<0}a^{w(\nu+\rho)-\rho}c_{\psi_{i,n}}(\nu, w)(a),
\]
where
\[
c_{\psi_{i,n}}(\nu,w)(a)=W_n(a^{-\alpha_i}, 1+w(\nu+\rho)(h_{\alpha_i})) \prod_{\substack{\alpha\in \Phi_+\cap w^{-1}\Phi_- \\ \alpha\neq -w^{-1}\alpha_i}}\frac{\xi(-(\nu+\rho)(h_\alpha))}{\xi(1-(\nu+\rho)(h_\alpha))}
\]
with $W_n(y,s)$ being defined by $($\ref{whittaker}$)$.
\end{theorem}

\begin{proof} As in the proof of Lemma \ref{vanishing}, we see that
\[
E_{\nu,\psi}(a)=\sum_{w\in W,\ w^{-1}\alpha_i<0}\int_{N_{w,\mathbb{A}}}\Phi_\nu(an_w w)\bar{\psi}(n_w)dn_w.
\]
We follow the computation in \cite{G04} and \cite[4.4]{Li}.
Let
$w^{-1}=r_{k_1}\cdots r_{k_{\ell}}$ be the reduced expression of $w^{-1}$, where $\ell=\ell(w)$ and
$k_j=1$ or $2$ for $j=1,\ldots, \ell$. Let
\[
\Phi_{w^{-1}}=\Phi_+\cap
w^{-1}\Phi_-=\{\beta_1,\ldots,\beta_{\ell}\},\] where $\beta_j= r_{k_1}\cdots
r_{k_{j-1}}\alpha_{k_j}.$ Then
\[
\Phi_{w}=\Phi_+\cap w\Phi_-=\{\gamma_1,\ldots,\gamma_{\ell}\},
\]
where
$\gamma_j=-w\beta_j=r_{k_{\ell}}\cdots r_{k_{j+1}}\alpha_{k_j}.$ Note that
\[
\beta_1+\cdots+\beta_{\ell}=\rho-w^{-1}\rho,\quad \gamma_1+\cdots+\gamma_{\ell}=\rho-w\rho.
\]
From these formulas it is clear that if $w^{-1}\alpha_i<0$, then $\gamma_{\ell}=\alpha_i=-w\beta_{\ell}$.
By decomposing $N_w$ into a product of root subgroups, we get
\begin{eqnarray*}
&&\int_{N_{w,\mathbb{A}}}\Phi_\nu(an_w w)\bar{\psi}(n_w)dn_w\\
&=&\int_{\mathbb{A}^{\ell}}\Phi_\nu \big(a\chi_{\gamma_{\ell}}(u_{\ell})\cdots \chi_{\gamma_1}(u_1)w\big)\bar{\psi}_{i,n}(u_{\ell})du_{\ell}\cdots du_1\\
&=&\int_{\mathbb{A}^{\ell}}\Phi_\nu \big(\chi_{\gamma_{\ell}}(a^{\gamma_{\ell}}u_{\ell})\cdots \chi_{\gamma_1}(a^{\gamma_1}u_1)aw\big)\bar{\psi}_{i,n}(u_{\ell})du_{\ell}\cdots du_1\\
&=&\int_{\mathbb{A}^{\ell}}a^{w\nu-\gamma_1-\cdots-\gamma_{\ell}}\Phi_\nu \big(\chi_{\gamma_{\ell}}(u_{\ell})\cdots \chi_{\gamma_1}(u_1)w\big)\bar{\psi}_{i,n}(a^{-\alpha_i}u_{\ell})du_{\ell}\cdots du_1\\
&=&a^{w(\nu+\rho)-\rho}\int_{\mathbb{A}^{\ell}}\Phi_\nu \big(\chi_{-\beta_{\ell}}(u_{\ell})\cdots \chi_{-\beta_1}(u_1)\big)\bar{\psi}_{i,n}(a^{-\alpha_i}u_{\ell})du_{\ell}\cdots du_1.
\end{eqnarray*}
Let $\chi_{-\beta_{\ell}}(u)=k(u)a(u)n(u)$ be the Iwasawa decomposition, with $k(u)\in K$, $n(u)\in U_{\beta_{\ell}}$, $a(u)\in A$.
Put $w'^{-1}=r_{k_1}\cdots r_{k_{\ell-1}}$, then $\{\beta_1,\ldots, \beta_{\ell-1}\}=\Phi_+\cap w'^{-1}\Phi_-$. Consider the decomposition
$N=N_{w'}N'_{w'}$ (see the proof of Lemma \ref{vanishing}). Then we have
\[
U_{-\beta_{\ell-1}}\cdots U_{-\beta_1}=w'^{-1}N_{w'}w'.
\]
Let us define the projection
\[
\pi: w'^{-1}Nw'\to w'^{-1}N_{w'}w'.
\]
Since $U_{-\beta_1},\ldots, U_{-\beta_{\ell-1}},$ $U_{\beta_{\ell}}\subset w'^{-1}Nw'$, the following map
\[
\pi\circ\textrm{Ad}(n(u)): w'^{-1}N_{w'}w'\to w'^{-1}N_{w'}w',
\]
is well-defined and unimodular. From this fact, and noting that $\Phi_\nu$ is right invariant under $w'^{-1}N'_{w'}w'\subset N$, it follows
\begin{eqnarray*}
&&\int_{\mathbb{A}^{\ell}}\Phi_\nu \big(\chi_{-\beta_{\ell}}(u_{\ell})\cdots \chi_{-\beta_1}(u_1)\big)\bar{\psi}_{i,n}(a^{-\alpha_i}u_{\ell})du_{\ell}\cdots du_1\\
&=&\int_{\mathbb{A}}a(u_{\ell})^{\nu+\beta_1+\cdots\beta_{\ell-1}}\bar{\psi}_{i,n}(a^{-\alpha_i} u_{\ell})du_{\ell}\int_{\mathbb{A}^{\ell-1}}\Phi_\nu \big(\chi_{-\beta_{\ell-1}}(u_{\ell-1})\cdots \chi_{-\beta_1}(u_1)\big)du_{\ell-1}\cdots du_1\\
&=&\int_{\mathbb{A}}a(u_{\ell})^{\nu+\rho-w'^{-1}\rho}\bar{\psi}_{i,n}(a^{-\alpha_i} u_{\ell})du_{\ell}\int_{\mathbb{A}^{\ell-1}}\Phi_\nu \big(\chi_{-\beta_{\ell-1}}(u_{\ell-1})\cdots \chi_{-\beta_1}(u_1)\big)du_{\ell-1}\cdots du_1.
\end{eqnarray*}
Note that $w'^{-1}\rho(h_{\beta_{\ell}})=\rho(w'h_{\beta_{\ell}})=\rho(h_{\alpha_i})=1$. Then the first integral in the last equation equals
$W_n(a^{-\alpha_i}, 1-(\nu+\rho)(h_{\beta_{\ell}}))=W_n(a^{-\alpha_i}, 1+w(\nu+\rho)(h_{\alpha_i}))$, and the second one, by Gindikin-Karpelevich formula, equals
\[
\prod^{\ell-1}_{j=1}\frac{\xi(-(\nu+\rho)(h_{\beta_j})\big)}{\xi\big(1-(\nu+\rho)(h_{\beta_j}))}.
\]
Hence our assertion follows.
\end{proof}


\section{Entirety of cuspidal Eisenstein series}

In analogy with \cite[Theorem 5.2]{BK} and \cite{GMP}, we shall prove that the Eisenstein series on $G$ induced from cusp forms on $SL_2$ are entire functions.
Let $P$ be the maximal parabolic subgroup of $G$ generated by $B$ and the simple reflection $r_1$, and $P=MU$ be the Levi decomposition, where
$M$ is the Levi subgroup and $U$ is the pro-unipotent radical of $P$. Let $L\cong SL_2$ be the subgroup generated by $\chi_{\pm \alpha_1}(t)$, $t\in \mathbb{R}$, and let $A_1=\langle h_{\alpha_1}(t): t\in\mathbb{R}^\times\rangle$, $H=\{a\in A: a^{\alpha_1}=\pm1\}$. Then we have an almost direct product $A=A_1\times H$, and $M=LH$.
We introduce
\begin{eqnarray*}
&&H^+=\{h_{\alpha_1}(t^m) h_{\alpha_2}(t^2): t\in\mathbb{R}_+\}\cong K\cap H\backslash H,\\
&&  A_1^+=\{h_{\alpha_1}(t): t\in\mathbb{R}_+\}\cong K\cap A_1\backslash A_1.
\end{eqnarray*}
From the Iwasawa decomposition $G=KP=KMU$, the following maps
\[
\textrm{Iw}_L: G\to K\cap L\backslash L, \quad \textrm{Iw}_{H^+}: G\to H^+
\]
are well-defined. Similarly, using the Iwasawa decomposition for $L$ we may define the map
\[
\textrm{Iw}_{A_1^+}: K\cap L\backslash L\to A_1^+. \]
For convenience we also denote by Iw$_{A^+}$ the map $G\to A^+,$ $g=k_ga_gn_g\mapsto a_g$ which we used previously. Then it is clear that
\begin{equation}\label{iwamap} \textrm{Iw}_{A^+}=(\textrm{Iw}_{A_1^+}\circ\textrm{Iw}_L)\times \textrm{Iw}_{H^+}: G\to A^+\cong A_1^+\times H^+.
\end{equation}

Let $\varpi_2$ be the fundamental weight corresponding to $\alpha_2$, that is
\[
\varpi_2= \frac{m\alpha_1+2\alpha_2}{4-m^2}.
\]
Note that $\varpi_2$ is trivial on $A_1=L\cap A$, hence (\ref{iwamap}) implies that
\begin{equation} \label{varpi2}
\textrm{Iw}_{A^+}(\cdot)^{\varpi_2}=\textrm{Iw}_{H^+}(\cdot)^{\varpi_2}.
\end{equation} Similarly, since $\alpha_1$ is trivial on $H$ we also have
\begin{equation}\label{alpha1}
       \textrm{Iw}_{A^+}(\cdot)^{\alpha_1}=\textrm{Iw}_{A_1^+}\circ\textrm{Iw}_L(\cdot)^{\alpha_1}.
\end{equation}
We may regard $\varpi_2$ as an algebraic character of $M$.
For $s\in\mathbb{C}$ define the Eisenstein series
\[
E_s(g)= \sum_{\gamma\in \Gamma/ \Gamma\cap P}\textrm{Iw}_{H^+}(g\gamma)^{s\varpi_2}.
\]
Moreover, for an unramified cusp form $f$ on $SL_2(\mathbb{R})/SL_2(\mathbb{Z})$, that is, a $SO(2)$--invariant cusp form, we define
\[
E_{s,f}(g)=\sum_{\gamma\in \Gamma/ \Gamma\cap P}\textrm{Iw}_{H^+}(g\gamma)^{s\varpi_2} f\big(\textrm{Iw}_L(g\gamma)\big).
\]

\begin{theorem}
Assume that $\mathrm{Re}~s<-2$. Then for any compact subset $A_c'$ of $A'$, there is a measure zero subset
$N_0$ of $N$ such that $E_s(g)$ converges absolutely for $g\in KA_c'N'$, where $N'=N-N_0$.
\end{theorem}

\begin{proof}
Our proof is similar to previous sections  where we induce Eisenstein series from characters on Borel subgroups.  We have the Bruhat decomposition
\[
G=\bigsqcup_{w\in W_1} BwP=\bigsqcup_{w\in W_1}N_wwP,
\]
where $W_1=\{w\in W: w\alpha_1>0\}$ is a set of representatives of minimal length for the quotient $W/\langle r_1\rangle$. Formal calculations show that the constant term of $E_s(a)$ for $a\in A'$ is
\[
E^\sharp_s(a)=\sum_{w\in W_1} a^{w(s\varpi_2+\rho)-\rho} c(s\varpi_2, w).
\]
Then we only need to prove the absolute convergence of $E^\sharp_s(a)$ under the conditions of the theorem.
We may assume that $s$ is real. There are two cases for $w\in W_1$: $w=(r_1r_2)^{n}$ or $r_2(r_1r_2)^n$. We shall only deal with the first case, and the second case
can be treated similarly.

Assume $w=(r_1r_2)^{n}$. Then using (\ref{w2}) and (\ref{wrho}), for $a\in A'$ there exists a positive constant $M$ only depending on $a$ such that
\[
a^{w(s\varpi_2+\rho)-\rho}\leq Ma^{\gamma C_n \alpha_1+C_n\alpha_2},
\]
where
\[
C_n=\frac{\gamma^{2n}}{(m^2-4)(\gamma+1)(\gamma-1)^2}\big[-s(m\gamma-2)(\gamma-1)-(m^2-4)\gamma\big].
\]
It follows that $C_n\to\infty$ if and only if
\begin{equation}\label{srange}
s<-\frac{(m^2-4)\gamma}{(m\gamma-2)(\gamma-1)}=-1-\gamma^{-1}.
\end{equation}

We also need to consider the factor $c(s\varpi_2, w)$. One can show that
\[
\Phi_+\cap w^{-1}\Phi_-=\{B_i \alpha_1 + B_{i+1}\alpha_2: i=0,\ldots, 2n-1\}.
\]
For $\alpha= B_i\alpha_1 + B_{i+1}\alpha_2$ we have
\[
-(s\varpi_2+\rho)(h_\alpha)=(-s-1)B_{i+1}-B_i>B_{i+1}-B_i\geq 1
\]
when $s<-2$. This implies that $c(s\varpi_2, w)\leq C^{\ell(w)}$ for some constant $C$ depending on $s$.

Combining above analysis it is easy to see the convergence of $E^\sharp_s(a)$ for $a\in A'$.
\end{proof}

We remark that, from $B_{i+1}>\gamma B_i$ it follows that
\[
(-s-1)B_{i+1}-B_i\to\infty \quad \textrm{as} \quad i\to\infty
\]
when $s<-1-\gamma^{-1}$. Together with (\ref{srange}), this suggests the following

\begin{conjecture}
$E_s(g)$ converges absolutely for $g\in KA'N$ and $\mathrm{Re}~s<-1-\gamma^{-1}$.
\end{conjecture}

Now let us state the main result of this section.

\begin{theorem} Let $f$ be an unramified cusp form on $SL_2$. For any compact subset $A_c'$ of $A'$, there is a measure zero subset $N_0$ of $N$ such that
$E_{s,f}(g)$ is an entire function of $s\in \mathbb{C}$ for $g\in KA_c'N'$, where $N'=N-N_0$.
\end{theorem}

We shall follow the strategy in \cite{GMP} to prove Theorem 7.3.  The following lemma is in analogy with \cite[Lemma 3.2]{GMP}, where we set $x^y:=yxy^{-1}$ for $x,y \in G$.

\begin{lemma}
If $\gamma\in \Gamma\cap BwB$, then
\[
\mathrm{Iw}_{A^+}(g\gamma)=\big(\mathrm{Iw}_{A^+}g\big)^{w^{-1}}\cdot \mathrm{Iw}_{A^+}(n_w w)
\]
for some $n_w\in N_{w,\mathbb{A}}$ depending on $\gamma$ and $g$.
\end{lemma}

Recall from \cite[Lemma 6.1]{GMS2} that for $n_w\in N_{w,\mathbb{A}}$,
\begin{equation} \label{iwa}
\ln\big(\mathrm{Iw}_{A^+}(n_w w)\big)=\sum_{\alpha\in \Phi_w} c_\alpha h_\alpha\quad\textrm{with}\quad c_\alpha\geq 0,
\end{equation}
where $\Phi_w=\Phi_+\cap w\Phi_-$. Now we can establish the Iwasawa inequalities:
\begin{lemma}
There exists a constant $D>0$ such that
\[
  \mathrm{Iw}_{A^+}(g\gamma)^{\alpha_1}\geq \mathrm{Iw}_{A^+}(g\gamma)^{D\varpi_2}
\]
for any $g\in KA'N$, $w\in W_1$ and $\gamma\in \Gamma\cap BwB$, where the constant $D$ is independent of $w$.
\end{lemma}

\begin{proof}
Put $a=\textrm{Iw}_{A^+}g\in A'$. From Lemma 7.4 and (\ref{iwa}), it suffices to find a constant $D$ such that
the following two inequalities
\begin{eqnarray}\label{iwa1}
 && a^{w^{-1}\alpha_1}\geq a^{Dw^{-1}\varpi_2},\\
  && \alpha_1(h_\alpha)\geq D\varpi_2(h_\alpha)  \label{iwa2}
\end{eqnarray}
hold for any $w\in W_1$ and $\alpha\in\Phi_w$.  We only consider the case $w=(r_1r_2)^{n}$, and similar arguments apply to the other case
$w=r_2(r_1r_2)^n$.  Put $w=(r_1r_2)^n$, then $\Phi_w=\{B_{i+1}\alpha_1+B_i\alpha_2 : i=0,\ldots, 2n-1\}$. For any $\alpha$ of the form $B_{i+1}\alpha_1+B_i\alpha_2$ one has \[
\alpha_1(h_\alpha)=2B_{i+1}-mB_i \geq (2\gamma-m)B_i= (2\gamma-m) \varpi_2(h_\alpha).
\]
This proves (\ref{iwa2}) with $D=2\gamma-m>0$.   

Now we prove (\ref{iwa1}). Again, there are two possibilities for $w\in W_1$: $w=(r_1r_2)^n$ or $r_2(r_1r_2)^n$. Here we treat the first case, while the second one is similar.
For $w=(r_1r_2)^n$, we see  from the formula \eqref{w2} that
\[
\left\{\begin{array}{l}
w^{-1}\alpha_1=-B_{2n-1}\alpha_1-B_{2n}\alpha_2,\\
w^{-1}\alpha_2=B_{2n}\alpha_1+B_{2n+1}\alpha_2.
\end{array}\right.
\]
Since we have
\[
\varpi_2=\frac{m\alpha_1+2\alpha_2}{4-m^2},
\]
it follows from the above formulas  that
\[
w^{-1}\varpi_2=\frac{1}{4-m^2}\left[ (2B_{2n}-mB_{2n-1})\alpha_1+(2B_{2n+1}-mB_{2n})\alpha_2\right].
\]
Recall that $B_n=\gamma^n+o(1)$, which implies
\[
2B_{n+1}-mB_{n}=(2\gamma-m)\gamma^{n}+o(1).
\]
 Note that $2\gamma-m>0$. Then one can find a positive constant $D$ independent of $n$ such that
 \begin{equation} \label{eqn-ab}
 B_n \geq D \, \frac{2B_{n+1}-m B_n}{m^2-4}.
 \end{equation}
In fact, for any $0<D<\frac{m^2-4}{2\gamma-m}$, the inequality \eqref{eqn-ab} is true as long as $n$ is large. Since $a\in A'$, we have $a^{\alpha_1}, a^{\alpha_2}<1$, and the above inequality \eqref{eqn-ab} implies that
\[
a^{w^{-1}\alpha_1} \geq a^{Dw^{-1}\varpi_2}.
\]
Note that we can choose the constant $D$
to be independent of $a\in A'$ and $w \in W_1$.
\end{proof}

We also need the rapid decay of cuspidal automorphic forms.\footnote{In an earlier version of this paper, an exponential decay in \cite{GMP} was used. It was pointed out by Steve D. Miller that the rapid decay was enough to obtain our result.} Most recent results on rapid decay, which generalize classical results in various ways,  can be found in \cite{MiSch}.  In our case, $f$ is an $SO(2)$-finite cusp form on $SL_2(\mathbb{R})/SL_2(\mathbb{Z})$, and we may assume that $g \in SL_2(\mathbb R)$ is in a Siegel set. Then the rapid decay implies that for any natural number $n\geq 1$, there exists a constant $C>0$ depending on $n$ such that  
\begin{equation}\label{decay}
|f(g)|\leq C \, \mathrm{Iw}_{A_1^+}(g)^{-n\alpha_1}.
\end{equation}

{\it Proof of Theorem 7.3.} Since cusp forms are bounded, the assertion follows from Theorem 7.1 when Re $s<-2$.  Assume that Re $s\geq -2$. Choose $s_0\in\mathbb{R}$
with $s_0<-2$. Choose a real number $d>\textrm{Re }s -s_0>0$ such that $\frac{d}{D}\in\mathbb{N}$, where $D$ is given in Lemma 7.5. Then Re $s-d< s_0 <-2$ and
there exists a subset $N'$
of $N$ with measure zero complement such that $E_{\textrm{Re }s-d}(g)$ converges for any $g\in KA'_cN'$.

Put $n=\frac{d}{D}\in\mathbb{N}$ as above. From (\ref{varpi2}), (\ref{alpha1}), Lemma 7.5 and (\ref{decay}), we obtain that for any $\gamma\in \Gamma/\Gamma\cap P$,
\begin{eqnarray*}
 &&\big| \textrm{Iw}_{H^+}(g\gamma)^{s\varpi_2} f\big(\textrm{Iw}_L(g\gamma)\big)\big|\\
 &\leq &C  \, \textrm{Iw}_{H^+}(g\gamma)^{(\textrm{Re }s)\varpi_2}
 \textrm{Iw}_{A_1^+}\circ\textrm{Iw}_L(g\gamma)^{-n\alpha_1}\\
 &\leq &   C \, \textrm{Iw}_{H^+}(g\gamma)^{(\textrm{Re }s)\varpi_2}\textrm{Iw}_{H^+}(g\gamma)^{-nD\varpi_2}\\
 &=& C  \, \textrm{Iw}_{H^+}(g\gamma)^{(\textrm{Re }s-d)\varpi_2}.
\end{eqnarray*}
Note that the constants $C$ and $D$ are independent of $w$. Taking the summation over $\gamma$, it follows that $E_{s,f}(g)$ is absolutely convergent.
\qed

\end{document}